\documentclass[11pt]{article}
\usepackage{amsmath,amsfonts,amsthm,amssymb, mathtools}
\usepackage{mathrsfs, graphicx,color,latexsym, tikz, calc,
}
\usepackage{tikz-cd}
\usepackage{graphicx}
\usepackage{epstopdf}
\usepackage{enumerate}
\usepackage{caption}
\usepackage{stmaryrd}
\usepackage{hyperref}
\usepackage{cite}

\usetikzlibrary{shadows}
\usetikzlibrary{patterns,arrows,decorations.pathreplacing}

\usepackage[margin=2.5cm]{geometry}
\usepackage{ulem}

\newtheorem{theorem}{Theorem}[section]

\newtheorem{lemma}[theorem]{Lemma}

\newtheorem{problem}[theorem]{Problem}
\newtheorem{conjecture}[theorem]{Conjecture}
\newtheorem{corollary}[theorem]{Corollary}

\newtheorem{prop}[theorem]{Proposition}
\newtheorem{claim}{Claim}[section]

\newtheorem{example}{Example}[section]

\allowdisplaybreaks[4]

\title{\bf \Large }
\date{\today }
\title{{\bf \Large 
The Suda–Tanaka-Tokushige conjecture for $\mathbf{p}$-biased  intersecting families}\footnote{ Lihua Feng was supported by the NSFC (Nos. 12271527 and 12471022) and NSF of Qinghai Province (No. 2025-ZJ-902T). E-mail addresses: \url{wuyjmath@163.com} (Y. Wu), \url{fenglh@163.com} (L. Feng). }
\author{
{\small  Yongjiang Wu,\ \ Lihua Feng\footnote{Corresponding author}
}\\[2mm]
\small School of Mathematics and Statistics, HNP-LAMA, Central South University\\
 \small Changsha, Hunan, 410083, China\\ 
}}

\begin{document}
\maketitle
\begin{abstract}  
In 2017, Suda, Tanaka and Tokushige conjectured that if $1>p_1\ge\cdots\ge p_n>0$ with $p_3\le \frac{1}{2}$, then every intersecting family $\mathcal A\subseteq 2^{[n]}$ satisfies $\mu_{\mathbf{p}}(\mathcal A)\le p_1$, where $\mu_{\mathbf{p}}$ is the non-uniform product measure defined by $\mu_{\mathbf{p}}(\mathcal{A})=\sum_{A\in\mathcal{A}} \prod_{i\in A} p_i \prod_{j\in [n]\setminus A}(1-p_j)$. In addition, if $p_1 > p_3$ or $p_1 < \frac{1}{2}$, then equality holds if and only if $\mathcal{A}$ is a  star centered at some $i \in [n]$ with $p_i = p_1$.
In this paper, we prove this conjecture in the following stronger $t$-intersecting form: for any $t\ge 1$, if $p_{t+2}\le \frac{1}{t+1}$, then  every $t$-intersecting family $\mathcal{A} \subseteq 2^{[n]}$ satisfies 
 $\mu_{\mathbf{p}}(\mathcal A)\le \prod_{i=1}^t p_i$. Moreover,  when $p_{t+2}<\frac{1}{t+1}$, equality holds if and only if  $\mathcal{A}=\{A\subseteq [n]: T\subseteq A\}$ for some $T\in \binom{[n]}{t}$ with $\prod_{i\in T} p_i=\prod_{i=1}^t p_i$. Our result unifies and generalizes the classical theorems of
Fishburn-Frankl-Freed-Lagarias-Odlyzko and Friedgut.
\end{abstract}



\section{Introduction}
For integers $a \leq b$, we denote $[a,b]=\{a,a+1,\dots,b\}$. Let $[n]:=[1,n]$ and let $2^{[n]}$ denote the power set of $[n]$. Two families $\mathcal{A},\mathcal{B}\subseteq 2^{[n]}$ are called \textit{cross-intersecting} if $A\cap B\neq\emptyset$ for all $A\in\mathcal{A}$, $B\in\mathcal{B}$. More generally, they are called \textit{cross-$t$-intersecting} if $|A\cap B|\ge t$ for all $A\in\mathcal{A}$, $B\in\mathcal{B}$. When $\mathcal{A}=\mathcal{B}$, the family is simply called \textit{$t$-intersecting} if $|A\cap A'|\ge t$ for all $A,\ A'\in\mathcal{A}$. In particular, when $t=1$, $\mathcal{A}$ is called \textit{intersecting}. A family of the form $\{A\subseteq[n]:T\subseteq A\}$ for some $T\in\binom{[n]}{t}$ is called a \textit{$t$-star}.  A $1$-star is simply called a \textit{star}.

Intersection theorems are among the central pillars of extremal set theory. The classical Erd\H{o}s--Ko--Rado theorem \cite{E61} states that for $n \ge 2k$, every intersecting family $\mathcal{A} \subseteq \binom{[n]}{k}$ satisfies $|\mathcal{A}| \le \binom{n-1}{k-1}$, with equality if and only if $\mathcal{A}=\big\{A\in \binom{[n]}{k}: x\in A\big\}$ for some $x \in [n]$, provided $n>2k$.
 This foundational result has inspired numerous generalizations, among which the $p$-biased measure formulation has emerged as a particularly natural and fruitful direction. For $0<p<1$, the \textit{$p$-biased measure} of a family $\mathcal{A}\subseteq 2^{[n]}$ is defined by
$$
\mu_p(\mathcal{A}) = \sum_{A\in\mathcal{A}} p^{|A|}(1-p)^{n-|A|},
$$
which assigns to each subset a weight proportional to its size. In this setting, stars again play the role of extremal families in an appropriate parameter range. 
Friedgut \cite{F08} developed a spectral approach to study intersecting families under this measure, showing that for $0<p\le \frac{1}{t+1}$, every $t$-intersecting family $\mathcal{A} \subseteq 2^n$ satisfies $\mu_p(\mathcal{A})\le p^t$, with equality characterized by $t$-stars when $p<\frac{1}{t+1}$. His proof introduces a pseudo-adjacency matrix of the disjointness graph whose eigenvalues can be explicitly computed.
A natural extension of the $p$-biased setting is to allow different probabilities for different elements of $[n]$. Let $\mathbf{p}=(p_1,\dots,p_n)$ with $p_i\in(0,1)$ for all $i\in[n]$. The \textit{$\mathbf{p}$-biased measure} is defined analogously by
$$
\mu_{\mathbf{p}}(\mathcal{A}) = \sum_{A\in\mathcal{A}} \prod_{i\in A} p_i \prod_{j\in [n]\setminus A}(1-p_j).
$$
This non-uniform version was introduced by Fishburn, Frankl, Freed, Lagarias and Odlyzko \cite{F86}, who proved that if $p_1 \ge p_2 \ge \cdots \ge p_n$ and $p_2\le \frac{1}{2}$, then every intersecting family $\mathcal{A} \subseteq 2^{[n]}$ satisfies $\mu_{\mathbf{p}}(\mathcal{A})\le p_1$, with equality only for the star centered at $1$ when $p_1$ is strictly maximal. Their result thus extends the classical EKR theorem to the non-uniform weighted setting, showing that stars remain extremal under a mild condition on the second-largest probability.
Suda,  Tanaka and Tokushige \cite{S17} used the semidefinite programming  to extend this result to cross intersecting setting for two families with non-uniform probabilities.
They also proposed the following conjecture for a single intersecting family, which refines the earlier result by relaxing the condition $p_2\le \frac{1}{2}$ to $p_3\le \frac{1}{2}$.

 \begin{conjecture}[Suda-Tanaka-Tokushige \cite{S17}]\label{S17}
Let $1 > p_1 \ge p_2 \ge \cdots \ge p_n > 0$ and $\mathbf{p}=(p_1,p_2,\ldots,p_n)$. 
Assume that  $p_3 \leq \frac{1}{2}$. 
If $\mathcal{A} \subseteq 2^{[n]}$ is intersecting, then
$$
\mu_{\mathbf{p}}(\mathcal{A}) \le p_1 .
$$
Moreover, if $p_1 > p_3$, or $p_1 < \tfrac12$, then equality holds if and only if $\mathcal{A}$ is a  star centered at some $i \in [n]$ with $p_1 = p_i$.
\end{conjecture}

Tokushige \cite{T22} proved that this conjecture holds when $p_1\leq \frac{1}{2}$ or $1-p_2>p_3$, using the Filmus–Golubev–Lifshitz high-dimensional Hoffman bound \cite{F21}. However, the full conjecture remains open.
In this paper, we prove  a natural generalization of the Suda–Tanaka–Tokushige conjecture to $t$-intersecting families. Our first main result is the following theorem.

 \begin{theorem}\label{Th2}
Let $1 > p_1 \ge p_2 \ge \cdots \ge p_n > 0$ and $\textbf{p}=(p_1,p_2,\ldots,p_n)$. 
For $t\geq 1$, assume that $p_{t+2} \leq \frac{1}{t+1}$. 
If $\mathcal{A} \subseteq 2^{[n]}$ is $t$-intersecting, then
$$
\mu_{\mathbf{p}}(\mathcal{A}) \le \prod_{i=1}^t p_i.
$$
Moreover, when $p_{t+2}  < \frac{1}{t+1}$,  equality holds if and only if $\mathcal{A}=\{A\subseteq [n]: T\subseteq A\}$ for some $T\in \binom{[n]}{t}$ with $\prod_{i\in T} p_i=\prod_{i=1}^t p_i$. 
\end{theorem}

The bound $p_{t+2}\le \frac{1}{t+1}$ in Theorem \ref{Th2} is best possible: if $p_{t+2}> \frac{1}{t+1}$, then the measure of a non-star $t$-intersecting family can already exceed that of the $t$-star. Indeed, when $p_1=\cdots=p_{t+2}=p$ with $\frac{1}{t+1}<p<1$, the triangular family
$
\mathcal{A}=\{A\subseteq[n]: |A\cap[t+2]|\ge t+1\}
$
satisfies $\mu_{\mathbf{p}}(\mathcal{A})=(t+2)p^{t+1}(1-p)+p^{t+2}>p^t$. 
On the other hand,
at the boundary $p_{t+2}=\frac{1}{t+1}$, the equality is attained by the same family when $p_1=\cdots=p_{t+1}=p$ and $p_{t+2}=1/(t+1)$, in which case $\mu_{\mathbf{p}}(\mathcal{A})=p^t=\prod_{i=1}^t p_i$ without $\mathcal{F}$ being a star. This explains why the equality statement in Theorem \ref{Th2} requires the strict condition $p_{t+2}<\frac{1}{t+1}$.

When $t=1$, Theorem \ref{Th2} reduces to the following result, which confirms Conjecture \ref{S17}.

\begin{corollary}\label{cr1}
Let $1 > p_1 \ge p_2 \ge \cdots \ge p_n > 0$ and $\textbf{p}=(p_1,p_2,\ldots,p_n)$. 
Assume that $p_3 \leq \frac{1}{2}$. 
If $\mathcal{A} \subseteq 2^{[n]}$ is intersecting, then
$$
\mu_{\mathbf{p}}(\mathcal{A}) \le p_1 .
$$
Moreover, if $p_3  < \frac{1}{2}$, then equality holds if and only if $\mathcal{A}$ is a  star centered at some $i \in [n]$ with $p_1 = p_i$.
\end{corollary}

When $ p_1 = \cdots = p_n = p$, Theorem  \ref{Th2} gives the following uniform version, which recovers Friedgut's  theorem.

\begin{corollary}\cite{F08}\label{cr2}
Let $0<p\leq \frac{1}{t+1}$. 
If $\mathcal{A} \subseteq 2^{[n]}$ is $t$-intersecting, then
$$
\mu_p(\mathcal{A}) \le p^t .
$$
Moreover, if $p  <  \frac{1}{t+1}$, then equality holds if and only if $\mathcal{A}$ is a $t$-star.
\end{corollary}

The proof of Theorem  \ref{Th2} is based on the shifting operator together with the generating  set method, which are introduced  in the following sections.
The remainder of this  paper is organized as follows. Section \ref{se2} introduces notation and preliminary tools. Section \ref{se3} presents the proof of Theorem  \ref{Th2}. Section \ref{se4} contains concluding remarks and open problems.

\section{Preliminaries}\label{se2}

\subsection{Notation}

Throughout this paper, we use the following notation. For a family $\mathcal{A}\subseteq 2^{[n]}$ and an element $x\in[n]$, we define
\begin{align*}
&\mathcal{A}^{(k)}=\{A\in\mathcal{A}: |A|=k\},~~\mathcal{A} [x]=\{A\in \mathcal{A}: x\in A\},\\
&\mathcal{A} (x)=\{A\backslash\{x\}: x\in A\in \mathcal{A}\},~~\mathcal{A}(\bar{x})=\{A\in\mathcal{A}: x\notin A\}.
\end{align*}
Thus, $\mathcal{A}[x]^{(k)}$ denotes the subfamily of $\mathcal{A}$ consisting of $k$-sets containing $x$, and $\mathcal{A}[x]^{(k)}(x)$ is the collection obtained by removing $x$ from each member of $\mathcal{A}[x]^{(k)}$.

A family $\mathcal{A}\subseteq 2^{[n]}$ is called an \textit{antichain} if there is no distinct $A, A'\in\mathcal{A}$ satisfying $A\subseteq A'$.
A family $\mathcal{A}\subseteq 2^{[n]}$ is called \textit{monotone} if $A\in\mathcal{A}$ and $A\subseteq B$ imply $B\in\mathcal{A}$. 
The \textit{up-set} of $\mathcal{A}$ is defined as
$$
\langle\mathcal{A}\rangle=\{F\subseteq[n]: A\subseteq F \text{ for some } A\in\mathcal{A}\}.
$$

\subsection{Shifting}

Shifting is a classical and powerful technique in extremal set theory. For a family $\mathcal{A}\subseteq 2^{[n]}$ and integers $1\le i<j\le n$, the  \textit{shifting operator} $s_{i,j}$ is defined by
$$
s_{i,j}(\mathcal{A})=\{s_{i,j}(A): A\in\mathcal{A}\},
$$
where
$$
s_{i,j}(A)=
\begin{cases}
(A\setminus\{j\})\cup\{i\}, & \text{if } j\in A,\ i\notin A,\ \text{and }(A\setminus\{j\})\cup\{i\}\notin\mathcal{A},\\
A, & \text{otherwise}.
\end{cases}
$$
A family is called \textit{shifted} if $s_{i,j}(\mathcal{A})=\mathcal{A}$ for all $1\le i<j\le n$. A well known fact, which follows by repeatedly applying the shifting operators, is that every family can be transformed into a shifted one.
The following key properties of shifting will be used.

\begin{lemma}\cite{G23}\label{G23}
If $\mathcal{A}\subseteq 2^{[n]}$ is $t$-intersecting, then $s_{i,j}(\mathcal{A})$ remains  $t$-intersecting for every $1\le i<j\le n$.
\end{lemma}

\begin{lemma}\label{le1}
Let $1 > p_1 \ge p_2 \ge \cdots \ge p_n > 0$ and $\textbf{p}=(p_1,p_2,\ldots,p_n)$.  Let $\mathcal{A} \subseteq 2^{[n]}$ and let $1\leq i<j\leq n$.
Then 
$ \mu_\textbf{p}(\mathcal{A})\le \mu_\textbf{p}(s_{ij}(\mathcal{A}))$.
\end{lemma}
\begin{proof}
For each $A \in \mathcal{A}$, we compare $\mu_{\mathbf{p}}(s_{ij}(A))$ with $\mu_{\mathbf{p}}(A)$.
If $s_{ij}(A)=A$, then clearly $\mu_{\mathbf{p}}(s_{ij}(A))=\mu_{\mathbf{p}}(A)$.
Otherwise, we have $j \in A$, $i \notin A$ and
$
s_{ij}(A) = (A\setminus\{j\}) \cup \{i\}.
$
Since $1>p_i \ge p_j > 0$, it follows that
$$
\frac{\mu_{\mathbf{p}}(s_{ij}(A))}{\mu_{\mathbf{p}}(A)}
= \frac{p_i(1-p_j)}{p_j(1-p_i)}\geq 1.
$$
 Hence, $\mu_{\mathbf{p}}(s_{ij}(A)) \ge \mu_{\mathbf{p}}(A)$. Summing over all $A \in \mathcal{A}$ yields
$
\mu_{\mathbf{p}}(s_{ij}(\mathcal{A}))
= \sum_{A \in \mathcal{A}} \mu_{\mathbf{p}}(s_{ij}(A))
\ge \sum_{A \in \mathcal{A}} \mu_{\mathbf{p}}(A)
= \mu_{\mathbf{p}}(\mathcal{A})
$.
\end{proof}

\subsection{Generating set}

We now introduce the generating set method, which originated from the work of Ahlswede and Khachatrian \cite{A97} and has since then become a key tool in the study of $t$-intersecting families; see also \cite{F17, L25}.
Let $\mathcal{A}$ be a monotone family. A \textit{generating set} of $\mathcal{A}$ is an inclusion-minimal member of $\mathcal{A}$. The collection of all generating sets is called the \textit{generating family} of $\mathcal{A}$, denoted by $\mathcal{G}$. Since no generating set contains another, $\mathcal{G}$ is an antichain.
The \textit{extent} of $\mathcal{A}$, denoted by $\ell$, is the largest integer appearing in any generating set. The \textit{boundary generating family} $\mathcal{G}[\ell]$ consists of all generating sets that contain $\ell$.

The following lemma is the main technical tool of the paper. It gives a precise formula for how the $\mathbf{p}$-biased measure changes when we replace, at the boundary layer, some generating sets containing $\ell$ by their deletions of $\ell$.

\begin{lemma}\label{le2}
Let $1 > p_1 \ge p_2 \ge \cdots \ge p_n > 0$ and $\mathbf{p} = (p_1, \dots, p_n)$.
Let $\mathcal{A}\subseteq 2^{[n]}$ be a monotone shifted family with extent $\ell\geq 2$, generating family $\mathcal{G}$ and boundary generating family $\mathcal{G}[\ell]$.  For any $\mathcal{B},\mathcal{C}\subseteq \mathcal{G}[\ell]$, define $\mathcal{D}=(\mathcal{G}\setminus(\mathcal{B}\cup\mathcal{C}) 
)\cup\{C\setminus\{\ell\}: C\in\mathcal{C}\}$.  Then
$$
\mu_{\mathbf{p}}(\langle \mathcal{D} \rangle) - \mu_{\mathbf{p}}(\mathcal{A})
=
\frac{1}{\prod_{k=\ell+1}^n (1-p_k)} \left(
\frac{1-p_\ell}{p_\ell}  \mu_{\mathbf{p}}( \mathcal{C}) - \mu_{\mathbf{p}}( \mathcal{B}\backslash \mathcal{C} )
\right).
$$
\end{lemma}

\begin{proof}
We begin with a structural description of the two disjoint parts of the symmetric difference between $\langle\mathcal{D}\rangle$ and $\mathcal{A}$. 

\begin{claim}\label{cl1}
Under the hypotheses of the lemma, we have
\begin{align*}
\langle\mathcal{D}\rangle\setminus\mathcal{A}=&\left\{(C\setminus\{\ell\})\cup T: C\in\mathcal{C},\ T\in\binom{[\ell+1,n]}{\leq n-|C|+1}\right\},\\
\mathcal{A}\setminus\langle\mathcal{D}\rangle=&\left\{B\cup T: B\in\mathcal{B}\backslash \mathcal{C},\ T\in\binom{[\ell+1,n]}{\leq n-|B|}\right\}.
 \end{align*}
\end{claim}
\begin{proof}[Proof of Claim \ref{cl1}]
We first prove the description of $\langle\mathcal{D}\rangle\setminus\mathcal{A}$.
For any $F\in\langle\mathcal{D}\rangle\setminus\mathcal{A}$, there exists some $C\in\mathcal{C}$ such that $C\setminus\{\ell\}\subseteq F$.  Since $C\subseteq [\ell]$, we have $C\setminus\{\ell\}\subseteq F\cap[\ell]$.
Moreover, $F\notin\mathcal{A}$  forces  $\ell \notin F\cap[\ell]$; otherwise $C\subseteq F\cap[\ell]$ would imply $F\in\mathcal{A}$, a contradiction.
If $C\setminus\{\ell\}\subsetneq F\cap[\ell]$, then there exists $i\in(F\cap[\ell])\setminus(C\setminus\{\ell\})$ with $i<\ell$. By shiftedness, $(C\setminus\{\ell\})\cup\{i\}\in\mathcal{A}$. Since 
$\mathcal{A}= \langle\mathcal{G}\rangle$, there exists $G\in\mathcal{G}$ such that $G\subseteq(C\setminus\{\ell\})\cup\{i\}\subseteq F\cap[\ell]$, which implies  $F\in\mathcal{A}$, a contradiction. 
Hence, $F\cap[\ell]=C\setminus\{\ell\}$. This implies that
$\langle\mathcal{D}\rangle\setminus\mathcal{A}\subseteq\left\{(C\setminus\{\ell\})\cup T: C\in\mathcal{C},\ T\in\binom{[\ell+1,n]}{\leq n-|C|+1}\right\}.$  
For the reverse, take any  $C\in\mathcal{C}$ and $ T\in\binom{[\ell+1,n]}{\leq n-|C|+1}$.
Since no subset of $C\setminus\{\ell\}$ belongs to $\mathcal{G}$,
 we have $(C\setminus\{\ell\})\cup T \notin \mathcal{A}$. Clearly, $(C\setminus\{\ell\})\cup T \in \langle\mathcal{D}\rangle$. 
Therefore, 
$\left\{(C\setminus\{\ell\})\cup T: C\in\mathcal{C},\ T\in\binom{[\ell+1,n]}{\leq n-|C|+1}\right\} \subseteq\langle\mathcal{D}\rangle\setminus\mathcal{A}$.

We now prove the  description of $\mathcal{A}\setminus\langle\mathcal{D}\rangle$.
For any $A\in\mathcal{A}\setminus\langle\mathcal{D}\rangle$, there exists some $B\in\mathcal{B}\backslash \mathcal{C}$ such that $B\subseteq A$. Then $B\subseteq A\cap[\ell]$.
Observe that $\ell\in B$. 
If $B\subsetneq A\cap[\ell]$, then there exists $i\in(A\cap[\ell])\setminus B$ with $i<\ell$.
By shiftedness, $(B\setminus\{\ell\})\cup\{i\}\in\mathcal{A}$. Since 
$\mathcal{A}= \langle\mathcal{G}\rangle$, there exists $G\in\mathcal{G}$ such that $G\subseteq(B\setminus\{\ell\})\cup\{i\}\subseteq A\cap[\ell]$.
In view of $\ell \notin G$, we have $G\in\mathcal{G}\setminus(\mathcal{B}\cup\mathcal{C})$.
Then $A\in\langle\mathcal{D}\rangle$, a contradiction. Hence, we have $A\cap[\ell]=B$, which implies 
$\mathcal{A}\setminus\langle\mathcal{D}\rangle\subseteq\left\{B\cup T: B\in\mathcal{B}\backslash \mathcal{C},\ T\in\binom{[\ell+1,n]}{\leq n-|B|}\right\}$.
For the reverse inclusion, take any  $B\in\mathcal{B}\backslash \mathcal{C}$ and $ T\in\binom{[\ell+1,n]}{\leq n-|B|}$. Note that no subset of $B$ belongs to $\mathcal{G}\setminus(\mathcal{B}\cup  \mathcal{C})$.
If  $B\cup T \in \langle\mathcal{D}\rangle$, then there exists some  
$C\in\mathcal{C}$ such that $C\setminus\{\ell\}\subseteq B$.
Since $\ell\in B$ and $B\in\mathcal{B}\backslash \mathcal{C}$, we  have $C\subsetneq B$. 
But $\mathcal{B}\cup\mathcal{C}$ is an antichain because $\mathcal{G}$ is a  generating family, a contradiction.
Hence, $B\cup T \notin \langle\mathcal{D}\rangle$, while clearly $B\cup T \in \mathcal{A}$. Consequently,
$\left\{B\cup T: B\in\mathcal{B}\backslash \mathcal{C},\ T\in\binom{[\ell+1,n]}{\leq n-|B|}\right\}\subseteq\mathcal{A}\setminus\langle\mathcal{D}\rangle$.
\end{proof}
Equipped with Claim \ref{cl1}, since $B, C \subseteq [\ell]$, we have   $n-|B| \ge n-\ell$ and $n-|C|+1 \ge n-\ell+1 > n-\ell$. Since $|[\ell+1,n]| = n-\ell$, the cardinality restrictions on $T$ in Claim \ref{cl1} are automatically satisfied for every $T \subseteq [\ell+1,n]$. Consequently, 
\begin{align}
\langle \mathcal{D} \rangle \setminus \mathcal{A} 
&= 
\left\{ (C \setminus \{\ell\}) \cup T : C \in \mathcal{C},\ T \subseteq [\ell+1,n] \right\}, \label{f1} \\
\mathcal{A} \setminus \langle \mathcal{D} \rangle 
&= 
\left\{ B \cup T : B \in \mathcal{B}\backslash \mathcal{C},\ T \subseteq [\ell+1,n] \right\}. \label{f2}
\end{align}
For any $S \subseteq [\ell]$, we have
\begin{align*}
\sum_{T \subseteq [\ell+1,n]} \mu_{\mathbf{p}}(S \cup T)
&=
\mu_{\mathbf{p}}^{[\ell]}(S) 
\sum_{T \subseteq [\ell+1,n]} 
\prod_{i \in T} p_i \prod_{j \in [\ell+1,n] \setminus T} (1-p_j) \\
&=
\mu_{\mathbf{p}}^{[\ell]}(S) 
\prod_{k=\ell+1}^n \big( p_k + (1-p_k) \big) =
\mu_{\mathbf{p}}^{[\ell]}(S),
\end{align*}
where $\mu_{\mathbf{p}}^{[\ell]}(S) := \prod_{i \in S} p_i \prod_{j \in [\ell]\setminus S} (1-p_j)$ denotes the restriction of $\mu_{\mathbf{p}}$ on $[\ell]$. Since $\mu_{\mathbf{p}}(S) = \mu_{\mathbf{p}}^{[\ell]}(S) \prod_{k=\ell+1}^n (1-p_k)$, we obtain the tail summation formula
\begin{equation}
\sum_{T \subseteq [\ell+1,n]} \mu_{\mathbf{p}}(S \cup T)
=
\frac{\mu_{\mathbf{p}}(S)}{\prod_{k=\ell+1}^n (1-p_k)}. \label{f3}
\end{equation}

For any $C \in \mathcal{C}$ and  $T \subseteq [\ell+1,n]$, we have
$
\mu_{\mathbf{p}}((C \setminus \{\ell\}) \cup T)
=
\frac{1-p_\ell}{p_\ell} \cdot \mu_{\mathbf{p}}(C \cup T). 
$
Combining this with (\ref{f1}), we obtain
\begin{align*}
\mu_{\mathbf{p}}(\langle \mathcal{D} \rangle \setminus \mathcal{A})
&=
\sum_{C \in \mathcal{C}} \sum_{T \subseteq [\ell+1,n]} 
\mu_{\mathbf{p}}((C \setminus \{\ell\}) \cup T) \\
&=
\frac{1-p_\ell}{p_\ell} 
\sum_{C \in \mathcal{C}} \sum_{T \subseteq [\ell+1,n]} 
\mu_{\mathbf{p}}(C \cup T).
\end{align*}
Applying (\ref{f3}) with $S = C$ yields
\begin{align*}
\mu_{\mathbf{p}}(\langle \mathcal{D} \rangle \setminus \mathcal{A})
&=
\frac{1-p_\ell}{p_\ell}
\sum_{C \in \mathcal{C}} 
\frac{\mu_{\mathbf{p}}(C)}{\prod_{k=\ell+1}^n (1-p_k)} \notag\\
&=
\frac{1-p_\ell}{p_\ell}
\frac{\mu_{\mathbf{p}}( \mathcal{C} )}{\prod_{k=\ell+1}^n (1-p_k)}.
\end{align*}
Similarly, using   (\ref{f2}) and (\ref{f3})  with $S = B$, we obtain
\begin{align*}
\mu_{\mathbf{p}}(\mathcal{A} \setminus \langle \mathcal{D} \rangle)
&=
\sum_{B \in \mathcal{B}\backslash \mathcal{C}} \sum_{T \subseteq [\ell+1,n]} 
\mu_{\mathbf{p}}(B \cup T) \\
&=
\sum_{B \in \mathcal{B}\backslash \mathcal{C}} 
\frac{\mu_{\mathbf{p}}(B)}{\prod_{k=\ell+1}^n (1-p_k)} \\
&=
\frac{\mu_{\mathbf{p}}(\mathcal{B}\backslash \mathcal{C} )}{\prod_{k=\ell+1}^n (1-p_k)}.
\end{align*}
It follows from 
$
\mu_{\mathbf{p}}(\langle \mathcal{D} \rangle) - \mu_{\mathbf{p}}(\mathcal{A})
=
\mu_{\mathbf{p}}(\langle \mathcal{D} \rangle \setminus \mathcal{A}) - \mu_{\mathbf{p}}(\mathcal{A} \setminus \langle \mathcal{D} \rangle)
$
that
$$
\mu_{\mathbf{p}}(\langle \mathcal{D} \rangle) - \mu_{\mathbf{p}}(\mathcal{A})
=
\frac{1}{\prod_{k=\ell+1}^n (1-p_k)} \left(
\frac{1-p_\ell}{p_\ell}  \mu_{\mathbf{p}}( \mathcal{C}) - \mu_{\mathbf{p}}( \mathcal{B}\backslash \mathcal{C})
\right).
$$
This completes the proof. 
\end{proof}

The next lemma records a basic structural property of boundary generating sets in a shifted $t$-intersecting family.

\begin{lemma}\label{le3}
 Let $\mathcal{A}\subseteq 2^{[n]}$ be  a monotone shifted $t$-intersecting family with  generating family  $\mathcal{G}$ and extent $\ell$. If 
$A, B\in \mathcal{G}[\ell]$ satisfy $|A\cap B|=t$, then $A\cup B=[\ell]$ and $|A|+|B|=t+\ell$.
\end{lemma}
\begin{proof}
Since $A, B\in \mathcal{G}[\ell]$, we have $A\cup B\subseteq[\ell]$ and $\ell\in A\cap B$. If $A\cup B\subsetneq[\ell]$, then there exists  $x\in[\ell]\setminus(A\cup B)$ with $x<\ell$. By shiftedness,  $(A\setminus\{\ell\})\cup\{x\}\in\mathcal{A}$ 
and $|((A\setminus\{\ell\})\cup\{x\})\cap B|=|(A\cap B)\setminus\{\ell\}|=t-1$, contradicting the fact that $\mathcal{A}$ is $t$-intersecting. Therefore, $A\cup B=[\ell]$ and $|A|+|B|=|A\cap B|+|A\cup B|=t+\ell$.
\end{proof}

Lemma \ref{le3} shows that if $A,B\in\mathcal{G}[\ell]$ satisfy $|A\cap B|=t$, then $A\cup B=[\ell]$ and $|A|+|B|=\ell+t$. This structural fact will be used in the constructions below. Depending on whether the boundary layers occur in two distinct sizes or only one, we obtain two different ways to reduce the extent of the family while preserving the $t$-intersecting property and controlling the measure.

\begin{lemma}\label{pr1}
Let $1 > p_1 \ge p_2 \ge \cdots \ge p_n > 0$ and $\mathbf{p} = (p_1, \dots, p_n)$.
Suppose that $\mathcal{A}\subseteq 2^{[n]}$ is  a monotone shifted $t$-intersecting family with  generating family  $\mathcal{G}$ and extent $\ell \geq 2$. Let $a\neq b$ be positive integers such that $a+b=t+\ell$ and $\mathcal{G}[\ell]^{(a)},\mathcal{G}[\ell]^{(b)}$ are not both empty. Define 
\begin{align*}
\mathcal{G}_1=\left(\mathcal{G}\setminus\left(\mathcal{G}[\ell]^{(a)}\cup\mathcal{G}[\ell]^{(b)}\right)\right)\cup\left\{H\setminus\{\ell\}: H\in\mathcal{G}[\ell]^{(a)}\right\},\\
\mathcal{G}_2=\left(\mathcal{G}\setminus\left(\mathcal{G}[\ell]^{(a)}\cup\mathcal{G}[\ell]^{(b)}\right)\right)\cup\left\{H\setminus\{\ell\}: H\in\mathcal{G}[\ell]^{(b)}\right\}.
\end{align*}
Let 
$\mathcal{A}_1=\langle\mathcal{G}_1\rangle$ and  $\mathcal{A}_2=\langle\mathcal{G}_2\rangle$.
Then both $\mathcal{A}_1$ and $\mathcal{A}_2$ are $t$-intersecting. Moreover, if $p_{\ell}\leq \frac{1}{2}$, then
$$
\max\{\mu_{\mathbf{p}}(\mathcal{A}_1),\mu_{\mathbf{p}}(\mathcal{A}_2)\}\geq \mu_{\mathbf{p}}(\mathcal{A}),
$$
with strict inequality  whenever $p_{\ell}<\frac{1}{2}$.
\end{lemma}

\begin{proof}
We first prove the $t$-intersecting property. By symmetry, it suffices to prove that $\mathcal{A}_1$ is $t$-intersecting. Since $\mathcal{A}_1=\langle\mathcal{G}_1\rangle$, it suffices to verify that $\mathcal{G}_1$ is $t$-intersecting.  Suppose, otherwise, that there exist $A,B\in\mathcal{G}_1$ such that $|A\cap B|\le t-1$. Since $\mathcal{G}$ is $t$-intersecting, at least one of  $A, B$ lies outside $\mathcal{G}$. Without loss of generality, assume $B\notin\mathcal{G}$. Then $B=C\setminus\{\ell\}$ for some $C\in\mathcal{G}[\ell]^{(a)}$.

If $\ell\notin A$, then $A\cap B=A\cap C$. Hence, $|A\cap C|\le t-1$.  Since  $\mathcal{G}$ is $t$-intersecting and $C\in\mathcal{G}$, it follows that $A\notin \mathcal{G}$. Thus, $A=D\setminus\{\ell\}$ for some $D\in\mathcal{G}[\ell]^{(a)}$. Since $2a\neq t+\ell$, Lemma \ref{le3} gives $|D\cap C|\geq t+1$. Consequently, $|A\cap B|=|D\cap C|-1\geq t$, a contradiction.

If $\ell\in A$, then $A\in\mathcal{G}[\ell]\setminus\left(\mathcal{G}[\ell]^{(a)}\cup\mathcal{G}[\ell]^{(b)}\right)$.  By Lemma \ref{le3}, we have $|A\cap C|\ge t+1$. But $|A\cap C|=|A\cap B|+1\le t$, again a contradiction.
Therefore, $\mathcal{G}_1$ is $t$-intersecting. 

It remains to prove the  measure inequality. By Lemma \ref{le2}, we have
\begin{align*}
\mu_{\mathbf{p}}(\mathcal{A}_1)-\mu_{\mathbf{p}}(\mathcal{A})
=&
\frac{1}{\prod_{k=\ell+1}^n(1-p_k)}
\left(
\frac{1-p_{\ell}}{p_{\ell}}\mu_{\mathbf{p}}\left(\mathcal{G}[\ell]^{(a)}\right)
-\mu_{\mathbf{p}}\left(\mathcal{G}[\ell]^{(b)}\right)
\right),\\
\mu_{\mathbf{p}}(\mathcal{A}_2)-\mu_{\mathbf{p}}(\mathcal{A})
=&
\frac{1}{\prod_{k=\ell+1}^n(1-p_k)}
\left(
\frac{1-p_{\ell}}{p_{\ell}}\mu_{\mathbf{p}}\left(\mathcal{G}[\ell]^{(b)}\right)
-\mu_{\mathbf{p}}\left(\mathcal{G}[\ell]^{(a)}\right)
\right).
\end{align*}
Adding these two inequalities yields
$$
\mu_{\mathbf{p}}(\mathcal{A}_1)+\mu_{\mathbf{p}}(\mathcal{A}_2)-2\mu_{\mathbf{p}}(\mathcal{A})
=
\frac{1}{\prod_{k=\ell+1}^n(1-p_k)}
\left(
\frac{1-p_{\ell}}{p_{\ell}}-1
\right)
\left(\mu_{\mathbf{p}}\left(\mathcal{G}[\ell]^{(a)}\right)+\mu_{\mathbf{p}}\left(\mathcal{G}[\ell]^{(b)}\right)
\right).
$$
Since $\frac{1-p_{\ell}}{p_{\ell}}-1=\frac{1-2p_{\ell}}{p_{\ell}}$,
the desired inequality follows immediately.
\end{proof}

\begin{lemma}\label{pr2}
Let $1 > p_1 \ge p_2 \ge \cdots \ge p_n > 0$ and $\mathbf{p} = (p_1, \dots, p_n)$.
Let $\mathcal{A}\subseteq 2^{[n]}$ be a monotone shifted $t$-intersecting family with  generating family  $\mathcal{G}$ and extent $\ell \geq 2$. Suppose that $a=\frac{t+\ell}{2}$ is a positive integer and  $\mathcal{G}[\ell]^{(a)}\neq \emptyset$. 
For each $i\in [\ell-1]$, define $\mathcal{H}_{i}=\{H\in \mathcal{G}[\ell]^{(a)}: i\in H \}$ and
\begin{align*}
\mathcal{G}_i=\left(\mathcal{G}\setminus\mathcal{G}[\ell]^{(a)}\right)\cup\left\{H\setminus\{\ell\}: H\in\mathcal{G}[\ell]^{(a)}\setminus \mathcal{H}_{i}\right\}.
\end{align*}
Let 
$\mathcal{A}_i=\langle\mathcal{G}_i\rangle$.
Then  each $\mathcal{A}_i$  is $t$-intersecting. Moreover, if $p_{\ell}\leq \frac{\ell-t}{2(\ell-1)}$, then there exists some $i\in [\ell-1]$ such that
$$
\mu_{\mathbf{p}}(\mathcal{A}_i)\geq \mu_{\mathbf{p}}(\mathcal{A}),
$$
with strict inequality  whenever $p_{\ell}<\frac{\ell-t}{2(\ell-1)}$.
\end{lemma}
\begin{proof}
We first prove that $\mathcal{A}_i$  is $t$-intersecting. It suffices to verify that $\mathcal{G}_i$ is $t$-intersecting, since $\mathcal{A}_i = \langle \mathcal{G}_i \rangle$. Suppose, to the contrary, that there exist $A, B \in \mathcal{G}_i$ with $|A \cap B| \le t-1$. Since $\mathcal{G}$ is $t$-intersecting, at least one of $A, B$ must lie outside $\mathcal{G}$. By symmetry, we may assume $B \notin \mathcal{G}$. Then $B = C \setminus \{\ell\}$ for some $C \in \mathcal{G}[\ell]^{(a)}\setminus \mathcal{H}_{i}$.

If $\ell \notin A$, then $A \cap B = A \cap C$, which yields $|A \cap C| \le t-1$. Since $C \in \mathcal{G}$ and $\mathcal{G}$ is $t$-intersecting, we obtain $A \notin \mathcal{G}$. Hence, $A = D \setminus \{\ell\}$ for some $D \in \mathcal{G}[\ell]^{(a)}\setminus \mathcal{H}_{i}$. Since $i\notin C\cup D$, we have $C\cup D\neq [\ell]$. 
It follows from Lemma \ref{le3} that $|D \cap C| \geq t+1$. Consequently,
$
|A \cap B| = |D \cap C| - 1 \ge t,
$
which condraticts the assumption.

If $\ell \in A$, then $A \in \mathcal{G}[\ell] \setminus \mathcal{G}[\ell]^{(a)}$. By Lemma \ref{le3}, we have $|A \cap C| \geq t+1$. However,
$
|A \cap C| = |A \cap B| + 1 \le t,
$
again a contradiction.
Therefore, $\mathcal{G}_i$ is $t$-intersecting. 

It remains to establish the measure inequality. By Lemma \ref{le2}, we have
\begin{align*}
\mu_{\mathbf{p}}(\mathcal{A}_i)-\mu_{\mathbf{p}}(\mathcal{A})&=\frac{1}{\prod_{k=\ell+1}^n(1-p_k)}
\left(
\frac{1-p_{\ell}}{p_{\ell}}\mu_{\mathbf{p}}\left(\mathcal{G}[\ell]^{(a)}\setminus \mathcal{H}_{i}\right)
-\mu_{\mathbf{p}}\left(\mathcal{H}_i\right)
\right)\\
&=\frac{1}{\prod_{k=\ell+1}^n(1-p_k)}
\left(
\frac{1-p_{\ell}}{p_{\ell}}\mu_{\mathbf{p}}\left(\mathcal{G}[\ell]^{(a)}\right)
-\frac{1}{p_{\ell}}\mu_{\mathbf{p}}\left(\mathcal{H}_i\right)
\right).
\end{align*}
Averaging over all $i \in [\ell-1]$ gives
$$
\sum_{i=1}^{\ell-1} \mu_{\mathbf{p}}(\mathcal{A}_i) - \mu_{\mathbf{p}}(\mathcal{A})
=\frac{1}{\prod_{k=\ell+1}^n(1-p_k)}
\left(
\frac{1-p_{\ell}}{p_{\ell}}\mu_{\mathbf{p}}\left(\mathcal{G}[\ell]^{(a)}\right)
-\frac{1}{p_{\ell}(\ell-1)}\sum_{i=1}^{\ell-1}\mu_{\mathbf{p}}\left(\mathcal{H}_i\right)
\right).
$$
Observe that every $H \in \mathcal{G}[\ell]^{(a)}$ contains $\ell$ and has size $a$,  so it is counted exactly $a-1$ times in $\sum_{i=1}^{\ell-1}\mu_{\mathbf{p}}(\mathcal{H}_i)$. Therefore,
$$
\sum_{i=1}^{\ell-1}\mu_{\mathbf{p}}(\mathcal{H}_i)
=
(a-1)\mu_{\mathbf{p}}\left( \mathcal{G}[\ell]^{(a)} \right).
$$
Recall that $t=2a-\ell$. If $p_{\ell}\leq \frac{\ell-t}{2(\ell-1)}$, then $1-p_{\ell}\geq 1-\frac{\ell-t}{2(\ell-1)}=\frac{a-1}{\ell-1}$, and hence
\begin{align*}
\sum_{i=1}^{\ell-1} \mu_{\mathbf{p}}(\mathcal{A}_i) - \mu_{\mathbf{p}}(\mathcal{A})
=&\frac{1}{\prod_{k=\ell+1}^n(1-p_k)}
\left(
\frac{1-p_{\ell}}{p_{\ell}}\mu_{\mathbf{p}}\left(\mathcal{G}[\ell]^{(a)}\right)
-\frac{a-1}{p_{\ell}(\ell-1)}\mu_{\mathbf{p}}\left( \mathcal{G}[\ell]^{(a)} \right)
\right)\geq 0.
\end{align*}
Consequently, there exist at least one index $i\in[\ell-1]$ such that $
\mu_{\mathbf{p}}(\mathcal{A}_i)\geq \mu_{\mathbf{p}}(\mathcal{A})$, with strict inequality  whenever $p_{\ell}<\frac{\ell-t}{2(\ell-1)}$.
This completes the proof.
\end{proof}

\begin{prop}\label{pr3}
Let $1 > p_1 \ge p_2 \ge \cdots \ge p_n > 0$ and $\mathbf{p} = (p_1, \dots, p_n)$.
Let $\mathcal{A}\subseteq 2^{[n]}$ be a monotone shifted $t$-intersecting family with  generating family  $\mathcal{G}$ and extent $\ell \geq 2$. 
\begin{itemize}

\item[(1)] Suppose that $\ell+t$ is odd and  $p_{\ell}\leq\frac{1}{2}$. 
Then there exists a $t$-intersecting family $\mathcal{A}'$ whose extent is at most $\ell-1$ such that 
$$
\mu_{\mathbf{p}}(\mathcal{A}')\geq \mu_{\mathbf{p}}(\mathcal{A}),
$$
with strict inequality  whenever $p_{\ell}<\frac{1}{2}$.

\item[(2)] Suppose that $\ell+t$ is even and  $p_{\ell}\leq\frac{\ell-t}{2(\ell-1)}$. 
Then there exists a $t$-intersecting family $\mathcal{A}^*$ whose extent is at most $\ell-1$ such that 
$$
\mu_{\mathbf{p}}(\mathcal{A}^*)\geq \mu_{\mathbf{p}}(\mathcal{A}),
$$
with strict inequality  whenever $p_{\ell}<\frac{\ell-t}{2(\ell-1)}$.
\end{itemize}
\end{prop}

\begin{proof}
 Since $\mathcal{G}[\ell]\neq \emptyset$, it follows that either $\mathcal{G}[\ell]^{(\frac{\ell+t}{2})}\neq\emptyset$, or
there exist distinct $a,b$ with $a+b=\ell+t$ such that $\mathcal{G}[\ell]^{(a)}$ and $\mathcal{G}[\ell]^{(b)}$  are not both empty.
We construct two subfamily $\mathcal{C}, \mathcal{D}\subseteq\mathcal{G}[\ell]$ as follows. 

Let $r=\frac{1-p_{\ell}}{p_{\ell}}$.
For any two distinct size $a, b$ with $a+b=\ell+t$, if $\mathcal{G}[\ell]^{(a)}$ and $\mathcal{G}[\ell]^{(b)}$ are not both empty, Lemma \ref{pr1} ensures that at least one of the two layers, say $\mathcal{G}[\ell]^{(a)}$, satisfies
$$
\mu_{\mathbf{p}}(\mathcal{A}_1)-\mu_{\mathbf{p}}(\mathcal{A})
=
\frac{1}{\prod_{k=\ell+1}^n(1-p_k)}
\left(
r\mu_{\mathbf{p}}\left(\mathcal{G}[\ell]^{(a)}\right)
-\mu_{\mathbf{p}}\left(\mathcal{G}[\ell]^{(b)}\right)
\right)\geq 0,
$$
with strict inequality  whenever $p_{\ell}<\frac{1}{2}$. We put $\mathcal{G}[\ell]^{(a)}$ into $\mathcal{C}$ and exclude $\mathcal{G}[\ell]^{(b)}$ from $\mathcal{C}$.

If $\ell+t$ is odd,  define
$$
\mathcal{G}'=(\mathcal{G}\setminus\mathcal{G}[\ell])\cup\{H\setminus\{\ell\}: H\in\mathcal{C}\},
$$
and set $\mathcal{A}'=\langle\mathcal{G}'\rangle$.
By Lemma \ref{le3}, $\mathcal{A}'$ is $t$-intersecting. By construction, $\mathcal{G}'$ contains no set containing $\ell$. Hence, the extent of $\mathcal{A}'$ is at most $\ell-1$.
Applying Lemma \ref{le2} with $\mathcal{B}=\mathcal{G}[\ell]$ and $\mathcal{C}=\mathcal{C}$ gives
\begin{align*}
\mu_{\mathbf{p}}(\mathcal{A}')-\mu_{\mathbf{p}}(\mathcal{A})
=
\frac{1}{\prod_{k=\ell+1}^n(1-p_k)}
\left(
r\mu_{\mathbf{p}}(\mathcal{C})
-
\mu_{\mathbf{p}}(\mathcal{G}[\ell]\setminus\mathcal{C})
\right)\geq 0,
\end{align*}
with strict inequality when $p_{\ell}<\frac{1}{2}$.

If $\ell+t$ is even, let $m=\frac{\ell+t}{2}$. If $\mathcal{G}[\ell]^{(m)}\neq \emptyset$,  then by Lemma \ref{pr2}, there exists some $i\in[\ell-1]$ for which, writing  $\mathcal{H}_i=\{H\in\mathcal{G}[\ell]^{(m)}: i\in H\}$,
\begin{align*}
\mu_{\mathbf{p}}(\mathcal{A}_i)-\mu_{\mathbf{p}}(\mathcal{A})=\frac{1}{\prod_{k=\ell+1}^n(1-p_k)}
\left(
r\mu_{\mathbf{p}}\left(\mathcal{G}[\ell]^{(m)}\setminus \mathcal{H}_{i}\right)
-\mu_{\mathbf{p}}\left(\mathcal{H}_i\right)
\right)\geq 0,
\end{align*}
with strict inequality  whenever $p_{\ell}<\frac{\ell-t}{2(\ell-1)}$.
Let $\mathcal{D}=\mathcal{C}\cup (\mathcal{G}[\ell]^{(m)}\setminus\mathcal{H}_i)$.
Now define
$$
\mathcal{G}^*=(\mathcal{G}\setminus\mathcal{G}[\ell])\cup\{H\setminus\{\ell\}: H\in\mathcal{D}\},
$$
and set $\mathcal{A}^*=\langle\mathcal{G}^*\rangle$.
By Lemma \ref{le3}, $\mathcal{A}^*$ is $t$-intersecting. Clearly, the extent of $\mathcal{A}^*$ is at most $\ell-1$.
Finally, Lemma \ref{le2} with $\mathcal{B}=\mathcal{G}[\ell]$ and $\mathcal{C}=\mathcal{D}$ gives
\begin{align*}
\mu_{\mathbf{p}}(\mathcal{A}^*)-\mu_{\mathbf{p}}(\mathcal{A})
=
\frac{1}{\prod_{k=\ell+1}^n(1-p_k)}
\left(
r\mu_{\mathbf{p}}(\mathcal{D})
-
\mu_{\mathbf{p}}(\mathcal{G}[\ell]\setminus\mathcal{D})
\right)\geq 0,
\end{align*}
with strict inequality when $p_{\ell}<\frac{\ell-t}{2(\ell-1)}$. 
\end{proof}

\section{Proof of Theorem \ref{Th2}}\label{se3}
Let $1 > p_1 \ge p_2 \ge \cdots \ge p_n > 0$ and $\textbf{p}=(p_1,p_2,\ldots,p_n)$. 
For $t\geq 1$, assume that $p_{t+2} \leq \frac{1}{t+1}$. 
Let $\mathcal{A} \subseteq 2^{[n]}$ be a $t$-intersecting family. 
By Lemmas \ref{G23} and \ref{le1}, repeated shifting operations transform $\mathcal{A}$ into a  shifted  $t$-intersecting family $\mathcal{A}'$ with $\mu_{\mathbf{p}}(\mathcal{A}')\ge \mu_{\mathbf{p}}(\mathcal{A})$. 
Hence, for the purpose of proving the upper bound, we may assume that $\mathcal{A}$  is  shifted. 
Since $\mu_{\mathbf{p}}(\langle\mathcal{A} \rangle) \geq \mu_{\mathbf{p}}(\mathcal{A})$ and $\langle \mathcal{A} \rangle$ is also $t$-intersecting, we may further assume that $\mathcal{A}$ is monotone.

 Let $\mathcal{G}$ be the generating family of $\mathcal{A}$  and let $\ell$ be its extent. Then $\mathcal{G}$
is $t$-intersecting and $\ell\geq t$.
We distinguish three cases according to the value of $\ell$.

\medskip
\textbf{Case 1: $\ell=t$.} 
Every set in $\mathcal{A}$ contains $[t]$. Then  $\mathcal{A}$ is contained in the $t$-star centered at $[t]$. Hence, $\mu_{\mathbf{p}}(\mathcal{A})\le \prod_{i=1}^t p_i$, with equality only when $\mathcal{A}$ is precisely that star.

\medskip
\textbf{Case 2: $\ell=t+1$.} 
By the definition of the boundary generating family, every set in $\mathcal{G}[t+1]$ contains $t+1$. If $\mathcal{G}[t+1]^{(t)}\neq \emptyset$, then, since  $\mathcal{A}$  is  shifted,  both  $[t]$ and $ [t-1]\cup \{t+1\}$ belong to $\mathcal{G}$, contradicting the fact that $\mathcal{G}$
is $t$-intersecting. As $\mathcal{G}$ is an antichain, the only remaining  possibility is $\mathcal{G}=\{[t+1]\}$. Consequently,  $\mu_{\mathbf{p}}(\mathcal{A})=\mu_{\mathbf{p}}(\langle \{[t+1]\} \rangle)= \prod_{i=1}^{t+1} p_i< \prod_{i=1}^t p_i$, since $p_{t+1}<1$.

\medskip
\textbf{Case 3: $\ell\ge t+2$.} 
If $\ell+t$ is odd, then applying Propsition \ref{pr3} (1)  yields a new $t$-intersecting family $\mathcal{A}'\subseteq 2^{[n]}$ satisfying
$\mu_{\mathbf{p}}(\mathcal{A}')\ge \mu_{\mathbf{p}}(\mathcal{A})$,
with  strict inequality  whenever $p_{\ell}<\frac{1}{2}$. Moreover, the extent of $\mathcal{A}'$  is strictly smaller  than  that of $\mathcal{A}$. Hence, by the induction hypothesis on the extent, we have $\mu_\textbf{p}(\mathcal{A}') \leq\prod_{i=1}^t p_i$. Consequently, $
\mu_\textbf{p}(\mathcal{A}) \le \prod_{i=1}^t p_i.
$
If $p_{t+2}  < \frac{1}{t+1}$, then $p_{\ell}<\frac{1}{2}$,
and therefore $
\mu_\textbf{p}(\mathcal{A}) < \prod_{i=1}^t p_i.
$

If $\ell+t$ is even, write $\ell=t+2r$. Then  $\frac{\ell-t}{2(\ell-1)}=\frac{r}{t+2r-1}$. Since  $\ell\ge t+2$, we have $r\geq 1$.
In view of
$
\frac{r}{t+2r-1}-\frac{1}{t+1}=\frac{(r-1)(t-1)}{(t+2r-1)(t+1)}\geq 0,
$
we have  $p_\ell\le p_{t+2} \leq \frac{1}{t+1}\leq \frac{\ell-t}{2(\ell-1)}$.
Applying Propsition \ref{pr3} (2), 
we obtain a new $t$-intersecting family $\mathcal{A}^*\subseteq 2^{[n]}$ satisfying
$\mu_{\mathbf{p}}(\mathcal{A}^*)\ge \mu_{\mathbf{p}}(\mathcal{A})$,
with  strict inequality  whenever $p_{\ell}<\frac{\ell-t}{2(\ell-1)}$. Moreover, $\mathcal{A}^*$  has smaller extent than $\mathcal{A}$. Hence, by the induction hypothesis on the extent, its measure is at most $\prod_{i=1}^t p_i$. 
Observe that if $p_{t+2}  < \frac{1}{t+1}$, then $p_{\ell}<\frac{\ell-t}{2(\ell-1)}$.
Therefore, $\mu_{\mathbf{p}}(\mathcal{A})\le \mu_{\mathbf{p}}(\mathcal{A}^*)\le \prod_{i=1}^t p_i$, with strict inequality in the case $p_{t+2}  < \frac{1}{t+1}$.

We now record the following lemma, which will be needed for the equality case.

\begin{lemma}\label{cla}
Let $\mathcal{A} \subseteq 2^{[n]}$ be a monotone $t$-intersecting family. 
Suppose that for some pair $i<j \in [n]$, the family 
$s_{i,j}(\mathcal{A})= \{ B \subseteq [n] : S\subseteq B \}$ for some $S\in \binom{[n]}{t}$.
Then $\mathcal{A} =\{ A \subseteq [n] : T \subseteq A \}$ for some $T\in \binom{[n]}{t}$.
\end{lemma}
\begin{proof}
Since the shifting operation preserves cardinality, we have
$
|\mathcal{A}| = |s_{i,j}(\mathcal{A})| = 2^{n-t}. 
$
Since $S \in s_{i,j}(\mathcal{A})$,  either $S \in \mathcal{A}$ or there exists $T \in \mathcal{A}$ such that $S=(T\backslash \{j\})\cup\{i\}\notin \mathcal{A}$.
If  $S \in \mathcal{A}$, then by monotonicity,  $
\{A \subseteq [n] : S\subseteq A \} \subseteq \mathcal{A}$.
Both sides have size $2^{n-t}$. Hence, 
$\mathcal{A}=\{A \subseteq [n] : S\subseteq A \}$.
If instead $T \in \mathcal{A} $, again by monotonicity,
$
\{ A \subseteq [n] : T \subseteq A \} \subseteq \mathcal{A}.
$
It follows from $
|\mathcal{A}| =  2^{n-t}. 
$ that
$
\mathcal{A} = \{ A \subseteq [n] : T \subseteq A \}.
$
\end{proof}

Now assume $p_{t+2}  < \frac{1}{t+1}$ and
$\mu_\textbf{p}(\mathcal{A})= \prod_{i=1}^t p_i$.
Then $\mathcal{A}$ and  $s_{i,j}(\mathcal{A})$ must be monotone.
If $\mathcal{A}$ is shifted, the preceding  argument forces  $\mathcal{A} = \{ A \subseteq [n] : [t]\subseteq A \}$. If $\mathcal{A}$ is not shifted, then Lemma  \ref{cla} imply that $\mathcal{A}=\{A\subseteq [n]: T\subseteq A\}$ for some $T\in \binom{[n]}{t}$. It follows that $\prod_{i\in T} p_i=\prod_{i=1}^t p_i$.
This completes the proof of Theorem \ref{Th2}.

\section{Concluding Remarks}\label{se4}

In this paper, we have investigated a non-uniform $\mathbf{p}$-biased measure problem for $t$-intersecting families. Our main result, Theorem \ref{Th2}, establishes that for any probability vector $\textbf{p}=(p_1,p_2,\ldots,p_n)$ satisfying 
$1>p_1\ge\cdots\ge p_n>0$ and $p_{t+2}\le\frac{1}{t+1}$, every $t$-intersecting family $\mathcal{A}\subseteq 2^{[n]}$ satisfies
$$
\mu_{\mathbf{p}}(\mathcal{A}) \le \prod_{i=1}^t p_i,
$$
with equality characterized by $t$-stars when $p_{t+2}<\frac{1}{t+1}$. This generalizes the classical result  of Fishburn-Frankl-Freed-Lagarias-Odlyzko \cite{F86} and Friedgut \cite{F08}, and confirms the Suda-Tanaka-Tokushige conjecture \cite{S17} as a special case $t=1$.
The proof relies on the shifting technique and the generating set method. This approach is purely combinatorial, in contrast to the spectral method of Friedgut \cite{F08} and Tokushige \cite{T22}, and the semidefinite programming method of Suda-Tanaka-Tokushige \cite{S17}.

A natural question is whether Theorem \ref{Th2} can extend to the cross $t$-intersecting setting. The straightforward analogue would require the same condition on the $(t+2)$-nd  largest probabilities for both families. However, this formulation fails. The following example, valid for every $t\ge 1$, demonstrates the obstruction.
\begin{example}\label{E1}
Let $n\geq t+2$ and let  $\textbf{p}=(p_1,p_2,\ldots,p_n)$,  $\textbf{q}=(q_1,q_2,\ldots,q_n)$ be probability vectors satisfying 
 $$1 > p_1 \ge p_2 \ge \cdots \ge p_{t+2}\geq p_j > 0, \ \ 1 > q_1 \ge q_2 \ge \cdots \ge q_{t+2}\geq q_j > 0, \ \ j\in [t+2, n].$$
 Take
$
p_1=\cdots=p_{t+1}=\frac{1}{2},\ p_{t+2}=\frac{1}{2(t+1)}$ and $
q_1=\cdots=q_{t+1}=\frac{3}{4},\ q_{t+2}=\frac{1}{2(t+1)}$.
Then $p_{t+2}=q_{t+2}=\frac{1}{2(t+1)}<\frac{1}{t+1}$.  Define
$$
\mathcal{A}=\{A\subseteq[n]:|A\cap[t+1]|\ge t\},\qquad
\mathcal{B}=\{B\subseteq[n]:[t+1]\subseteq B\}.
$$
These families are cross $t$-intersecting. A direct computation gives
$
\mu_{\mathbf{p}}(\mathcal{A})=\frac{t+2}{2^{t+1}},\ 
\mu_{\mathbf{q}}(\mathcal{B})=\left(\frac34\right)^{t+1}.
$
Hence,
$$
\mu_{\mathbf{p}}(\mathcal{A})\mu_{\mathbf{q}}(\mathcal{B})
=
\frac{t+2}{2^{t+1}}\left(\frac34\right)^{t+1}
>
\left(\frac38\right)^t
=
\prod_{i=1}^t p_iq_i.
$$
\end{example}

Example \ref{E1} suggests that a valid cross version requires the stronger condition
$
p_{t+1}\le \frac{1}{t+1},~ q_{t+1}\le \frac{1}{t+1}.
$
The following  lemma provides a reduction to  the uniform  setting under this strengthened assumption.

\begin{lemma}\label{c42}
Let $\mathcal{A}\subseteq 2^{[n]}$ be monotone. Let $\textbf{p}=(p_1,p_2,\ldots,p_n)$ be a probability vector satisfying 
 $$1 > p_1 \ge p_2 \ge \cdots \ge p_{t+1}\geq p_j > 0,  \ \ j\in [t+2, n].$$
Put $p=p_{t+1}$. Then
$$
\frac{\mu_{\mathbf p}(\mathcal{A})}{\prod_{i=1}^t p_i}
\le
\frac{\mu_p(\mathcal{A})}{p^t}.
$$
\end{lemma}
\begin{proof}
Let
$$
R(\mathbf p)=\frac{\mu_{\mathbf p}(\mathcal{A})}{\prod_{i=1}^t p_i}.
$$
For $i\le t$,  write
$
\mu_{\mathbf p}(\mathcal{A})=p_i u+(1-p_i)v,
$
where
$
u=\mu_{\mathbf p}^{[n]\backslash \{i\}}(\mathcal{A}(i))$, $ 
v=\mu_{\mathbf p}^{[n]\backslash \{i\}}(\mathcal{A}(\bar{i})) 
$
and $\mu_{\mathbf{p}}^{[\ell]}$ denotes the restriction of $\mu_{\mathbf{p}}$ on $[\ell]$.
Since $\mathcal{A}$ is monotone, we have $\mathcal{A}(\bar{i})\subseteq \mathcal{A}(i)$. Then
$u\ge v$.
The part of $R(\mathbf p)$ depending on $p_i$ is
$$
\frac{p_i u+(1-p_i)v}{p_i}
=
u-v+\frac{v}{p_i}.
$$
This is non-increasing as a function of $p_i$. Since $p_i\ge p$ for $i\le t$, decreasing $p_i$ to $p$ cannot decrease $R(\mathbf p)$.
For $j>t$, the denominator $\prod_{i=1}^t p_i$ does not contain $p_j$. Write
$
\mu_{\mathbf x}(\mathcal{A})=p_j x+(1-p_j)y.
$
Then $x\ge y$ by monotonicity. Hence, $\mu_{\mathbf p}(\mathcal {A})$ is non-decreasing in $p_j$. Since $p_j\le p$ for $j>t$, increasing $p_j$ to $p$ cannot decrease $R(\mathbf p)$.

After performing these operations for all coordinates, every coordinate becomes $p$, and $R(\mathbf p)$ has not decreased at any step. Therefore,
$
R(\mathbf p)\le \frac{\mu_p(\mathcal{A})}{p^t}.
$
\end{proof}

By Lemma \ref{c42}, the strengthened condition
$
p_{t+1}\le \frac{1}{t+1}, q_{t+1}\le \frac{1}{t+1}
$
admits a clean reduction of the non-uniform measure problem to the uniform unequal-biased cross $t$-intersecting product measure setting.
Let $\mathbf p=(p_1,\dots,p_n)$ and $\mathbf q=(q_1,\dots,q_n)$ be probability vectors satisfying
$$
1>p_1\ge p_2\ge\cdots\ge p_{t+1}\ge p_j>0,\qquad
1>q_1\ge q_2\ge\cdots\ge q_{t+1}\ge q_j>0
$$
for all $j\in[t+2,n]$. Let $\mathcal A,\mathcal B\subseteq 2^{[n]}$ be cross $t$-intersecting, and assume
$
p_{t+1}\le \frac{1}{t+1},\ q_{t+1}\le \frac{1}{t+1}.
$
Under these assumptions, the  maximization of $\mu_{\mathbf p}(\mathcal A)\mu_{\mathbf q}(\mathcal B)$ reduces  to  
that of  $\mu_{p_{t+1}}(\mathcal A)\mu_{q_{t+1}}(\mathcal B)$, which is precisely Tokushige's cross $t$-intersecting conjecture \cite{T13}, see \cite{W25} for recent progress. Moreover, the same reduction applies to the non-uniform measures of $r$-cross intersecting families. Consequently, several results in the uniform setting immediately yield their non-uniform counterparts via a simple and unified argument.

Beyond the cross intersecting setting, the regime $p_{t+2}>\frac{1}{t+1}$ in  Theorem \ref{Th2}  remains largely unexplored. The triangular families suggest that Frankl-type families may be extremal, pointing toward a possible non-uniform generalization of the weighted Ahlswede-Khachatrian theorem \cite{ F17}.
This leads naturally to the  following problem.

 \begin{problem}\label{pb1}
Let $1 > p_1 \ge p_2 \ge \cdots \ge p_n > 0$ and $\textbf{p}=(p_1,p_2,\ldots,p_n)$. 
Suppose that $\mathcal{A} \subseteq 2^{[n]}$ is $t$-intersecting. Determine the maximum of $\mu_{\mathbf{p}}(\mathcal{A})$. In particular, is it true that
$$
\mu_{\mathbf{p}}(\mathcal{A}) \le \max_{r: t+2r\le n}\mu_{\mathbf{p}}(\mathcal{F}(t,r)),
$$ 
where $\mathcal{F}(t,r)=\{F\subseteq [n]: |F\cap[t+2r]|\geq t+r\}$.
\end{problem}

\section*{Declaration of competing interest}
We declare that we have no conflict of interest to this work.

\section*{Data availability}
No data was used for the research described in the article.


\end{document}